\documentclass[final,12pt]{elsarticle}
\usepackage{xcolor,soul,cancel}
\usepackage[color]{showkeys}
\definecolor{refkey}{rgb}{0,1,1}
\definecolor{labelkey}{rgb}{1,0,0}
\usepackage{amssymb}
\usepackage{amsmath}
\usepackage{amsthm}
\usepackage{graphicx}
\usepackage{float}
\journal{arXiv}

\newtheorem{thm}{Theorem}

\newtheorem{cor}{Corollary}
\newtheorem{prop}[thm]{Proposition}
\newtheorem{Ex}{Example}
\numberwithin{equation}{section}
\newcommand{\eq} [1] {\begin{equation}\label{#1}\quad}
\newcommand{\en} {\end{equation}}
\newcommand{\scal}[1]{\langle#1\rangle}
\newcommand{\norm}[1]{\left\Vert#1\right\Vert}
\newcommand{\abs}[1]{\left\vert#1\right\vert}
\newcommand{\ceil}[1]{\left\lceil#1\right\rceil}
\newcommand{\floor}[1]{\left\lfloor#1\right\rfloor}

\newcommand{\C}{\mathbb C}
\newcommand{\M}{{\bf M}}

\renewcommand{\Re}{\operatorname{Re}}
\newcommand{\im}{\operatorname{Im}}

\newcommand{\conv}{\operatorname{conv}}
\newcommand{\re}{\operatorname{Re}}

\begin{document}

\begin{frontmatter}
%\ead{ims2@nyu.edu, ilya@math.wm.edu, imspitkovsky@gmail.com}

\title{On Kippenhahn curves and higher-rank \\ numerical ranges of some matrices \tnoteref{support}}
%----------Author 1

%\ead{ims2@nyu.edu, ilya@math.wm.edu, imspitkovsky@gmail.com}
\author[MUC]{Nat\'alia Bebiano}
\address[MUC]{Departamento de Matem\'atica, Universidade da Coimbra, Portugal}
\ead{bebiano@mat.uc.pt}
\author[FUC]{Jo\'ao da Provid\'encia}
\address[FUC]{Departamento de F\'isica, Universidade da Coimbra, Portugal}
\ead{providencia@uc.pt}
\author[nyuad]{Ilya M. Spitkovsky}
\ead{ims2@nyu.edu, ilya@math.wm.edu, imspitkovsky@gmail.com}
\address[nyuad]{Division of Science and Mathematics, New York  University Abu Dhabi (NYUAD), Saadiyat Island,
P.O. Box 129188 Abu Dhabi, United Arab Emirates}
\tnotetext[support]{The work of the first author [NB] was supported by the Centre for Mathematics of the University of Coimbra ---
UIDB/00324/2020, funded by the Portuguese Government through FCT/MCTES. The third author [IMS] was supported in part by Faculty Research funding from the Division of Science and Mathematics, New York University Abu Dhabi.}

\begin{abstract}
The higher rank numerical ranges of generic matrices are described in terms of the components of their Kippenhahn curves. Cases of tridiagonal (in particular, reciprocal) 2-periodic matrices are treated in more detail.
\end{abstract}

\end{frontmatter}

\section{Introduction} 

Let $\M_n$ stand for the algebra of all $n$-by-$n$ matrices with the entries $a_{ij}\in\C$, $i,j=1,\ldots n$. We will identify $A\in\M_n$ with a linear operator acting on $\C^n$, the latter being equipped with the standard scalar product $\scal{.,.}$ and the associated norm $\norm{x}:=\scal{x,x}^{1/2}$. The {\em numerical range} of $A$ is defined as \eq{nr} W(A)=\{\scal{Ax,x}\colon \norm{x}=1 \}, \en see e.g. \cite[Chapter 1]{HJ2} or more recent \cite[Chapter~6]{DGSV}  for the basic properties of $W(A)$, in particular its convexity and invariance under unitary similarities. 

In \cite{ChoKriZy06}, this notion was generalized as follows: the {\em rank-$k$ numerical range} of $A$ is \eq{knr} \Lambda_k(A)= \{ \lambda\in\C\colon PAP=\lambda P \text{ for some rank-}k \text{ orthogonal projection } P\}. \en 

Of course, \eq{chain} W(A)=\Lambda_1(A)\supseteq\Lambda_2(A)\supseteq \cdots\supseteq \Lambda_n(A). \en For $k>n/2$ the set $\Lambda_k(A)$ is empty or a singleton $\{\lambda_0\}$; in the latter case $\lambda_0$ is an eigenvalue of $A$ having geometric multiplicity at least $2k-n$ \cite[Proposition 2.2]{ChoKriZy06}.  In particular, $\Lambda_n(A)\neq\emptyset$ if and only if $A$ is a scalar multiple of the identity, and then all the sets in \eqref{chain} coincide. % with $\{\lambda_0\} 

So, for $k=1$ and $k>n/2$ the sets $\Lambda_k(A)$ are convex. Their convexity for intermediate values of $k$ was established in \cite{Woe08}. Shortly thereafter, in \cite{LiSze08} it was shown that, moreover,
\eq{knrint}\Lambda_k(A)=\bigcap_{\theta\in [0,2\pi)} \{ \mu\in\C\colon \Re (e^{i\theta}\mu)\leq\lambda_k(\theta)\},  \en
where $\lambda_k(\theta)$ stands for the $k$-th largest  (counting the multiplicities) eigenvalue of the matrix $\Re(e^{i\theta}A)$.
As usual, for any $X\in\M_n$ 
 \[ \re X=\frac{X+X^*}{2},\quad  \im X=\frac{X-X^*}{2i}.\]

When applied to normal matrices, \eqref{knrint} yields 
\eq{norm} \Lambda_k(N)=\cap\conv\{\lambda_{j_1},\ldots,\lambda_{j_{n-k+1}}\}, \en
with the intersection taken over all $(n-k+1)$-tuples from the spectrum $\sigma(N)$ of a normal matrix $N$.  This result is also from \cite{LiSze08}, confirming a conjecture from \cite{CHKZ}. 

Our next observation is that the boundary lines \eq{ell} \ell_{\theta,k}= \{ \mu\in\C\colon \Re (e^{i\theta}\mu)=\lambda_k(\theta)\} \en
of the half-planes in the right hand side of \eqref{knrint}, when taken for all $k=1,\ldots,n$, form a family the envelope of which is the so called {\em Kippenhahn curve} $C(A)$ of the matrix $A$. It was shown in \cite{Ki} (see also the English translation \cite{Ki08}) that $W(A)=\conv C(A)$. From the discussion above it is clear that, at least in principle, not only $W(A)$ but all the rank-$k$ numerical ranges of $A$ can be described in terms of $C(A)$. 

Section~\ref{s:gen} is devoted to generic matrices, for which $C(A)$ splits into $\ceil{n/2}$ components, each solely responsible for the respective higher rank numerical range. These results are specified further in Section~\ref{s:tri} for the case of tridiagonal 2-periodic matrices, when explicit formulas for $\lambda_k(\theta)$ are known. Finally, a particular case of reciprocal 2-periodic matrices is treated in Section~\ref{s:rec}.

\section{Generic matrices}\label{s:gen}

For $n=2$, there are only two sets in the chain \eqref{chain}, both easily identifiable. If $n=3$, the middle term is either a singleton or the empty set (since $2>3/2$). The next proposition allows to distinguish between the two possibilities.
\begin{prop}Let $A\in\M_3$. Then $\Lambda_2(A)\neq\emptyset$ if and only if $W(A)$ is an elliptical disk, possibly degenerating into a line segment.  \label{th:k2n3} \end{prop}
\begin{proof}Directly from the definition it follows that $\Lambda_2(A)$ is a singleton $\{\lambda\}$ if and only if $A$ is unitarily similar to $\begin{bmatrix}\lambda & 0 & x \\ 0 & \lambda & y \\ u & v & z\end{bmatrix}$. Applying another unitary similarity if needed, we may without loss of generality suppose that $u=0$.

{\sl Case 1.} $x=0$. Then $A=(\lambda)\oplus B$, where  $B=\begin{bmatrix}\lambda & y\\ v & z\end{bmatrix}$, and $W(A)=W(B)$ is either an elliptical disk or a line segment, depending on whether or not $B$ is normal. 

{\sl Case 2.} $x\neq 0$. Then $A$ is unitarily similar to the tridiagonal matrix $\begin{bmatrix}\lambda & x & 0 \\ 0 & z & v \\ 0 & y & \lambda\end{bmatrix}$ with $(1,2)$- and $(2,1)$-entries having distinct absolute values. According to \cite[Lemma 8]{BS041}, $A$ is unitarily irreducible. On the other hand, its $(1,1)$- and $(3,3)$-entries coincide, which implies the ellipticity of $W(A)$ \cite[Theorem 4.2]{BS04}.
\end{proof} 	
	
Recall that a matrix $A\in\M_n$ is {\em generic} if $\lambda_1(\theta),\ldots,\lambda_n(\theta)$ are distinct for all $\theta$. 

Normal matrices are not generic; for $n=2$ the converse is also true. Hence, there is a direct relation with the shape of the numerical range: $A\in\M_2$ is generic if and only if $W(A)$ is a non-degenerate elliptical disc. Already for $n=3$, things get more subtle. 
\begin{prop}\label{th:gen3}Let $A\in\M_3$. Then $A$ is generic if and only if $W(A)$:\\ {\rm (i)}   has an ovular shape, or\\  {\rm (ii)}  is an ellipse with no eigenvalues of $A$ lying on its boundary.  \end{prop} 
Note that $A$ is unitarily irreducible in case (i) while it may or may not be unitarily reducible (though not normal) in case (ii). 
\begin{proof} If $A$ is unitarily irreducible, according to \cite[Proposition 3.2]{KRS} it is generic if and only if $W(A)$ has no flat portions on the boundary. These are exactly ovular and elliptical shapes, as per Kippenhahn's classification. Moreover, unitary irreducibility of $A$ implies that its eigenvalues are not on the boundary. 

Normal matrices are not generic, as was mentioned earlier. In the remaining case, $W(A)$ is the convex hull of an ellipse $E$ and a normal eigenvalue $\lambda$ of $A$. The matrix is generic if $\lambda$ lies in the interior of $E$, which falls under (ii), and non-generic otherwise. \end{proof} 	

Comparing Propositions~\ref{th:k2n3} and \ref{th:gen3}, we see that for $A\in\M_3$ non-empty and empty $\Lambda_2(A)$ materialize both for generic and non-generic matrices. 

\begin{Ex}\label{ex1}  Let

$$
M_1=
\left[\begin{matrix}
0&-1/2& 0\\
2&0&-1/2\\
0&1/2& \sqrt2\\
\end{matrix}\right],~~
M_2=
\left[\begin{matrix}
0&1/2& 0\\
1/2&0&2\\
0&1& 0\\
\end{matrix}\right].~~
$$
Figure~\ref{fig2} refers to the matrix $M_1$ and Figure~\ref{fig3} refers to the matrix $M_2$. Observe that $W(M_1)$ is ovular,
$\Lambda_2(M_1)=\emptyset$, while $W(M_2)$ is elliptical and $\Lambda_2(M_2)=\{0\}$ is the eigenvalue of $M_2$ different from the foci $\pm 3/2$ of $W(M_2)$.

\begin{figure}[h]
\centering
\includegraphics[width=0.5\linewidth,angle=0]{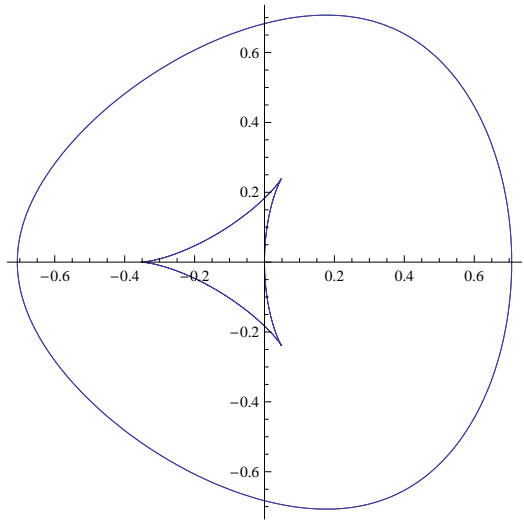}
%{ellipticity.eps}
\caption{Kippenhahn curve of $M_1$}
\label{fig2}
\end{figure}

\begin{figure}[h]
\centering
\includegraphics[width=0.6\linewidth,angle=0]{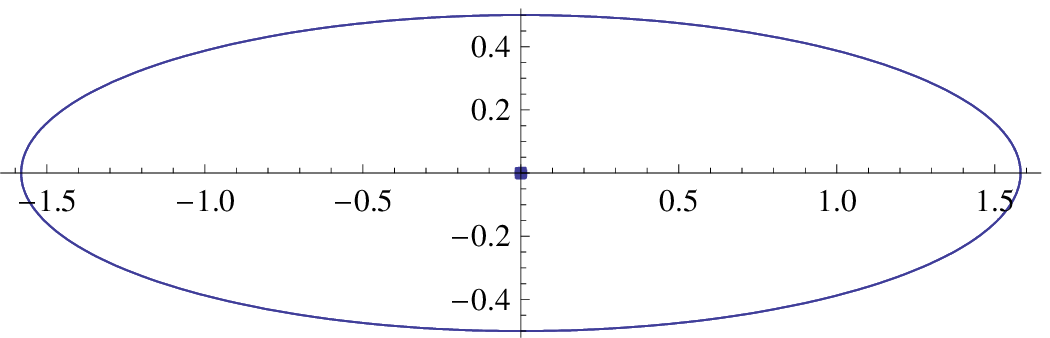}
%{ellipticity.eps}
\caption{Kippenhahn curve of $M_2$}
\label{fig3}
\end{figure}
\end{Ex}

%%%%%%%%%%%%%%%%%%%%%%%%%

Returning to generic matrices of arbitrary dimension $n$, note that from their definition it immediately follows that
\eq{op} \lambda_k(\theta)=-\lambda_{n-k+1}(\theta+\pi),\quad k=1,\ldots,n. \en
Since $\lambda_{n-k+1}(\theta)>\lambda_k(\theta)$ for $k>\ceil{n/2}$, the half-planes corresponding to $\theta$ and $\theta+\pi$ in \eqref{knrint} are disjoint. Therefore, the rank-$k$ numerical ranges of generic matrices $A$ are empty for $k>\ceil{n/2}$. On the other hand, directly from \eqref{knrint} we see that for generic matrices $A$ the inclusions in \eqref{chain} are proper for $k=1,\ldots,\ceil{n/2}$; moreover, $\Lambda_{k+1}(A)$ lies in the interior of $\Lambda_k(A)$.  

The structure of $C(A)$ and the related description of $\Lambda_k(A)$ for $k\leq\ceil{n/2}$ are as follows. 

\begin{thm}\label{th:Agen}For a generic matrix  $A\in\M_n$ its Kippenhahn curve $C(A)$ consists of the closed components
	\[ \gamma_k(A)=\{\scal{Az_k(\theta),z_k(\theta)}\colon \theta\in [0,2\pi]\}, \quad k=1,\ldots,\ceil{n/2}, \]
	where $z_k(\theta)$ is the unit eigenvector associated with the eigenvalue $\lambda_k(\theta)$ of $\re(e^{i\theta}A)$. Respectively, the half-planes in the representation \eqref{knrint} of $\Lambda_k(A)$ are bounded by the family \eqref{ell} of the tangent lines of $\gamma_k(A)$. \end{thm}
The first statement is a rewording (in different terms) of \cite[Theorem 13]{JAG98}, based in particular on \eqref{op}; the second immediately follows from the first. 

For $n$ odd and $k=\ceil{n/2}$ from \eqref{knrint}, \eqref{op} it can be seen that in fact $\Lambda_k(A)$ is the intersection of the tangent lines $\ell_{\theta,k}$ to $\gamma_k(A)$ defined by \eqref{ell}. This yields the following test for distinguishing between $\Lambda_{\ceil{n/2}}$ being a singleton or the empty set.  
\begin{cor}\label{th:midodd}Let $A\in\M_n$ be generic. If $n$ is odd, then $\Lambda_{\ceil{n/2}}(A)=\gamma_{\ceil{n/2}}(A)$ 
if $\gamma_{\ceil{n/2}}(A)$ is a point, and $\Lambda_{\ceil{n/2}}(A)=\emptyset$ otherwise.  \end{cor} 
\noindent Both cases are illustrated by Example~\ref{ex1}. 

Corollary~\ref{th:midodd} implies that for odd $n$ the curve  $\gamma_{\ceil{n/2}}(A)$ cannot be convex unless it collapses to a single point. On the other hand, the outermost curve $\gamma_1(A)$ of $C(A)$ for a generic matrix $A$ is always convex, and thus coincides with the boundary $\partial W(A)$ of its numerical range.  This means in particular that $\partial W(A)$ does not have corners or flat portions. Other components of $C(A)$ may exhibit cusps and swallowtails but no inflection points. 

As can be seen from Fig.~\ref{fig2}, cusps (but not swallowtails) materialize already when $n=3$. The emergence of swallowtails will be demonstrated in Section~\ref{s:rec}, see Fig.~\ref{fig4by4}--\ref{fig7N}. 

Convexity of $\gamma_1(A)$ implies that the subsequent components lie strictly inside of it. This, however, does not  preclude $\gamma_j(A)$ with $j>1$ from intersecting, as soon as there are at least two of them (i.e., when $n\geq 5$ -- see Fig.~\ref{fig00} in Section~\ref{s:tri} for an example corresponding to $n=5$).  Note that this is happening in spite of strict inclusions in \eqref{chain}. 	
\section{Tridiagonal 2-periodic matrices}\label{s:tri} 
A matrix $A\in\M_n$ is {\em tridiagonal}  if $a_{ij}=0$ whenever $\abs{i-j}>1$.  We will be making use of the well known (and easy to prove) recursive relation for the determinants $\Delta_n$ of such matrices,
\eq{rec} \Delta_n=a_{nn}\Delta_{n-1}-a_{n-1,n}a_{n,n-1}\Delta_{n-2}, \en 
implying in particular that $\Delta_n$ is invariant under transpositions $a_{i+1,i}\leftrightarrow a_{i,i+1}$ of its off-diagonal pairs. 

Suppose now that these pairs are {\em unbalanced}, i.e., \eq{abs} \abs{a_{i+1,i}}\neq\abs{a_{i,i+1}} \text{ for } i=1,\ldots,n-1.\en  Then hermitian matrices $\re(e^{i\theta}A)$ will be {\em proper} tridiagonal, i.e., their entries directly above and below the main diagonal will be non-zero. According to \cite[Corollary 7]{BS041}, the eigenvalues of $\re(e^{i\theta}A)$ are simple for all $\theta$, thus implying the genericity of $A$. 

\begin{Ex}\label{5x5}
Let % the $5\times5$ tridiagonal matrix
	$$
	M_3=
	\left[\begin{matrix}
		1&1& 0&0&0\\
		1/4&2&1/2&0&0\\
		0&1/4&0&3/4&0\\
		0&0&1/4&-2&1\\
		0&0&0&1/4&-1
		%0&0&0&0&1/4
\end{matrix}\right]. $$ \end{Ex}

This matrix is generic, since \eqref{abs} holds. According to Corollary~\ref{th:midodd}, $\Lambda_3=\emptyset$.
\iffalse The characteristic polynomial is given by
	\begin{eqnarray*}&&f(z,\theta)=\frac{1}{4096}((478 - 768 z^2) \cos\theta + 
		6 z (-1059 + 1984 z^2) \cos2 \theta\\&& + (215 - 128 z^2) \cos3 \theta - 
		8 z (589 - 1672 z^2 + 512 z^4 + 221 \cos4 \theta) + 28 \cos5 \theta).
	\end{eqnarray*}
	The boundary generating curve is given in Figure \ref{fig00} \fi 
	\begin{figure}[h]
		\centering
		\includegraphics[width=0.5\linewidth,angle=0]{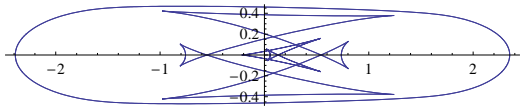}
		%{ellipticity.eps}
		\caption{Kippenhahn curve of $M_3.$ Notice that $\gamma_2$ intersects $\gamma_3.$}
		\label{fig00}
	\end{figure}

We will say that a tridiagonal matrix $A$ is {\em 2-periodic} if so are the sequence of its diagonal entries and of its (non-ordered) off-diagonal pairs. For such matrices we will use the notation $a_1,a_2$ for the first two diagonal entries, and $\{b_1,c_1\}, 
\{b_2,c_2\}$ for the first two (once again, non-ordered) pairs of the off-diagonal entries. 

Along with $A$, for any $\theta$ the hermitian matrix $\re(e^{i\theta}A)$ will be 2-periodic as well, with $\alpha_j(\theta)=:\re(e^{i\theta}a_j)$ ($j=1,2$) as the period of its main diagonals. Transposing their off-diagonal pairs as needed, we may arrange for the superdiagonal to also be 2-periodic, with \eq{beta}\beta_j(\theta)=:(e^{i\theta}b_j+e^{-i\theta}\overline{c_j)}/2, \quad j=1,2\en  as the first two entries. According to \eqref{rec}, this rearrangement preserves the characteristic polynomial of $\re(e^{i\theta}A)$. Therefore, explicit formulas from \cite{Gov94} can be used to compute $\lambda_k(\theta)$ in our setting. The respective straightforward computation shows that 
\eq{lt}\lambda_{k,n-k+1}=\frac{\alpha_1+\alpha_2}{2}\pm\sqrt{\left(\frac{\alpha_1-\alpha_2}{2}\right)^2+\abs{\beta_1}^2+\abs{\beta_2}^2
+2\abs{\beta_1\beta_2}Q_k}\en 
for $k=1,\ldots,m:=\floor{n/2}$, while $\lambda_{m+1}=\alpha_1$ if $n$ is odd.

Here $Q_k=\cos\frac{k\pi}{m+1}$ if $n$ is odd, and the $k$-th (in the decreasing order) root of the $m$-th degree polynomial $q_m$ defined recursively via \eq{qrec} q_0=1,\  q_1(\mu)=\mu+\abs{\beta_2/\beta_1},  q_{k+1}(\mu)=\mu q_k(\mu)-q_{k-1}(\mu)\text{ for } k\geq 1 \en if $n$ is even. 
 
For odd $n$, directly from the formula for $\lambda_{m+1}$ we obtain 
\begin{prop}\label{th:lastodd}Let $A\in\M_n$ be tridiagonal and 2-periodic. If $n$ is odd, then $\gamma_{\ceil{n/2}}(A)=\{a_1\}$, the {\rm (1,1)}-entry of $A$. \end{prop} 
According to Corollary~\ref{th:midodd}, for such matrices $\Lambda_{\ceil{n/2}}(A)=\{a_1\}$. Also, by Proposition~\ref{th:lastodd} a 2-periodic tridiagonal matrix $A\in\M_5$ cannot have intersecting $\gamma_2$ and $\gamma_3$. For $n=6$, however, this becomes a possibility; see Fig.~\ref{fig6N} in Section~\ref{s:rec}.
 
The parameters $Q_k$ are explicit and constant when $n$ is odd, and implicit (and in general depending on $\theta$) if $n$ is even. This makes consideration of even-sized matrices much harder. However, in the case \eq{AAS} \overline{b_1}c_2=c_1\overline{b_2}\en  treated in \cite{AAS}, the ratio $\abs{\beta_2/\beta_1}$ is the same as $\abs{b_2/b_1}$ and thus $\theta$-independent. According to \eqref{qrec}, $Q_k$ then do not depend on $\theta$ for even $n$ as well. Formulas \eqref{lt}, with some addtional nontrivial computations, provide an alternative approach to the complete description of  rank-$k$ numerical ranges of 2-periodic tridiagonal matrices satisfying \eqref{AAS}. In agreement with \cite{AAS}, they all happen to be elliptical disks. 

Condition \eqref{AAS} holds in particular for tridiagonal Toeplitz matrices. If in addition either the super- or the subdiagonal vanishes, then the dependence on $\theta$ disappears in \eqref{lt} alltogether. In other words, $\gamma_k$ are then concentric circles, and $\Lambda_k(A)$ the respective circular disks. This covers the result on shift operators from \cite{Gaa}.  

\begin{Ex} To illustrate other possible shapes of Kippenhahn curves for 2-periodic tridiagonal matrices, let $M_4\in\M_7$ have the zero main diagonal and $b_1=3, b_2=6, c_1=c_2=2$. \end{Ex} 
\begin{figure}[h]
\centering
\includegraphics[width=0.9\linewidth,angle=0]{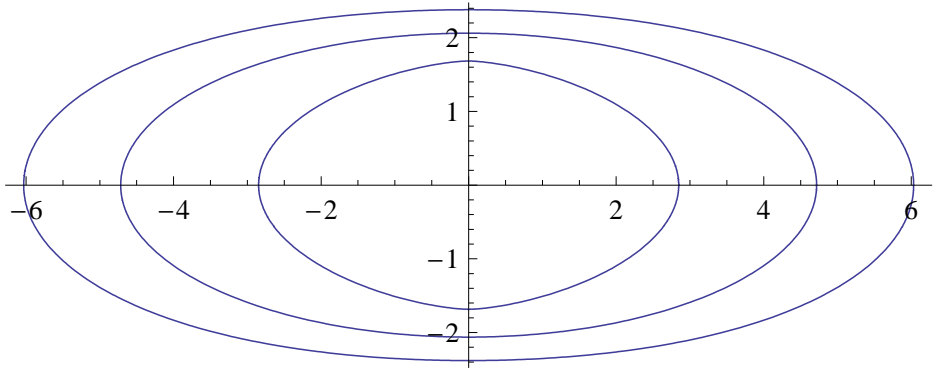}
%{figure6.eps}
%%{ellipticity.eps}
\caption{Kippenhahn curve of $M_4$} %   $7\times7$ 2-Toeplitz matrix, Non reciprocal,~~upper diagonal: $a_o=3,a_e=6$ lower diagonal: $a_o=a_e=2$. }
%}}
\label{figure7NR}
\end{figure}

See the next section for more specific examples.

\section{Reciprocal matrices} \label{s:rec}
	
Recall the notion of reciprocal matrices introduced in \cite{BPSV}. These are tridiagonal matrices with constant (without loss of generality, zero) main diagonal and the off diagonal pairs satisfying $a_{i+1,i}a_{i,i+1}=1$.  Reciprocal matrices are of course proper tridiagonal. Denoting $\abs{a_{j+1,j}}^2+\abs{a_{j,j+1}}^2=:2A_j$ we see that $A_j\geq 1$. Condition \eqref{abs} for such matrices takes the form $A_j>1$, $j=1,\ldots,n-1$. 

A 2-periodic reciprocal matrix $A$ is completely characterized by its size $n$ and the values $a_1:=\abs{a_{12}}, a_2:=\abs{a_{23}}$ (alternatively, by $A_1$ and $A_2$). For $n\geq 4$ (the only interesting setting), $\im A$ has multiple eigenvalues if $A_1$ or $A_2$ is equal to one, and so conditions $A_1,A_2>1$ are not only sufficient but also necessary for $A$ to be generic. 

Moreover, for reciprocal matrices \eqref{beta} yields $\abs{\beta_j}=\sqrt{(A_j+\tau)/2}$, where $\tau=\cos(2\theta)$. So, according to \eqref{lt} 
$\lambda_{k,n-k+1}$ in this case are the square roots of 
\eq{zeta} \zeta_k=\frac{1}{2}(A_1+A_2+2\tau)+\sqrt{(A_1+\tau)(A_2+\tau)}Q_k, \quad j=1,\ldots,m. \en

Observe that the right hand side of \eqref{zeta} is invariant under the substitutions $\theta\mapsto -\theta$ and $\theta\mapsto\theta+\pi$. Thus, we arrive at the following

\begin{cor}\label{co:sym}Let $A\in\M_n$ be a 2-periodic reciprocal matrix. Then each component $\gamma_1,\ldots\gamma_m$  of its Kippenhahn curve $C(A)$, and consequently its rank-$k$ numerical ranges $\Lambda_k(A)$ for $k=1,\ldots,m$, are symmetric with respect to both horizontal and vertical coordinate axes. Also, $\gamma_{m+1}=\Lambda_{m+1}=\{0\}$ if $n$ is odd.  \end{cor}

Furthermore, $\gamma_k$ is an ellipse if and only if $\zeta_k=x\tau+y$ with some constant $y>x>0$. 
If $A_1=A_2:=A$, this happens to be the case for all $k$, since then
\[ \zeta_k= (A+\tau)(1+Q_k),\]
with $Q_k$ constant (note that \eqref{AAS} holds in a trivial way). So, the rank-$k$ numerical ranges of such matrices are elliptical disks with the boundaries $\{\gamma_k\}_{k=1}^m$ forming a family of nested ellipses whose axes are coincident with the coordinate axes.  

On the contrary, when $A_1\neq A_2$ we have  

\begin{thm}\label{th:onel}Let $A$ be a 2-periodic reciprocal matrix of odd size $n$ and $A_1\neq A_2$. Then none of its rank-$k$ numerical ranges has an elliptical shape if $n=1\mod 4$. Otherwise, exactly one of them, namely $\Lambda_{(n+1)/4}(A)$, is an elliptical disk.  \end{thm}
\begin{proof}The first summand in the right hand side of \eqref{zeta} is of desired form. The second term, however, is such only if $Q_k=0$. Since $Q_k=\cos\frac{k\pi}{m+1}$ for odd $n$, the result follows. 
\end{proof}
Observe that for generic 4-by-4 matrices $\gamma_1$ and $\gamma_2$ (consequently, $\Lambda_1$ and $\Lambda_2$) are elliptical only simultaneously. Recall also that the numerical range of a reciprocal matrix $A\in\M_4$ is elliptical if and only if 
\eq{gold} A_2=\phi A_1-\phi^{-1}A_3 \text{ or } A_2=\phi A_3-\phi^{-1}A_1, \en
where $\phi$ is the golden ratio, and at least one of the inequalities $A_j\geq 1$ is strict \cite[Theorem 7]{BPSV}. If $A$ in addition is 2-periodic, i.e. $A_1=A_3$, then \eqref{gold} implies $A_2=A_1$. In other words, neither of rank-$k$ numerical ranges of such $A$ is elliptical, unless $A_1=A_2$. 

We suspect that this is the case for generic 2-periodic reciprocal matrices $A\in\M_n$ for all even $n>2$, not just $n=4$. Formulas \eqref{zeta} should be instrumental in proving this conjecture; the difficulty lies in the implicit nature of $Q_k$ for even values of $n$. 
	
Kippenhahn curves of several reciprocal matrices are pictured below. The matrices are described by the triples $\{n,\abs{a_1},\abs{a_2}\}$, or $\{n,A_1,A_2\}$. In Fig.~\ref{fig6x6}, \ref{fig6N} and \ref{fig7NC}, the dotted curves are the best fitting ellipses to the components of $C(A)$ which look elliptical but in fact are not.

\begin{figure}[p]
	\centering
	\includegraphics[width=0.75\linewidth,angle=0]{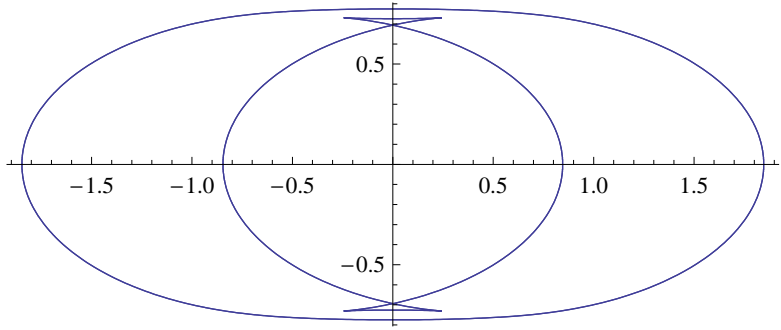}
	\caption{$n=4,a_1=2,a_2=21/20$. The numerical range $\Lambda_1$ is bounded by the exterior component, while  $\Lambda_2$ is bounded by the interior component with its swallowtails removed; $\Lambda_3=\emptyset$.}
	\label{fig4by4}
\end{figure}

\begin{figure}[p]
	\centering
	\includegraphics[width=0.75\linewidth,angle=0]{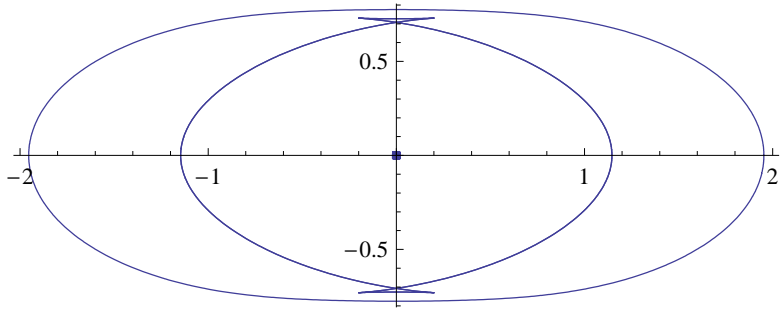}
	\caption{$n=5,a_1=2,a_2=21/20$. The picture is similar to Fig.~\ref{fig4by4}, except that now $\Lambda_3=\{0\}$.}
	\label{fig5by5}
\end{figure}

\begin{figure}[p]
\centering
\includegraphics[width=0.9\linewidth,angle=0]{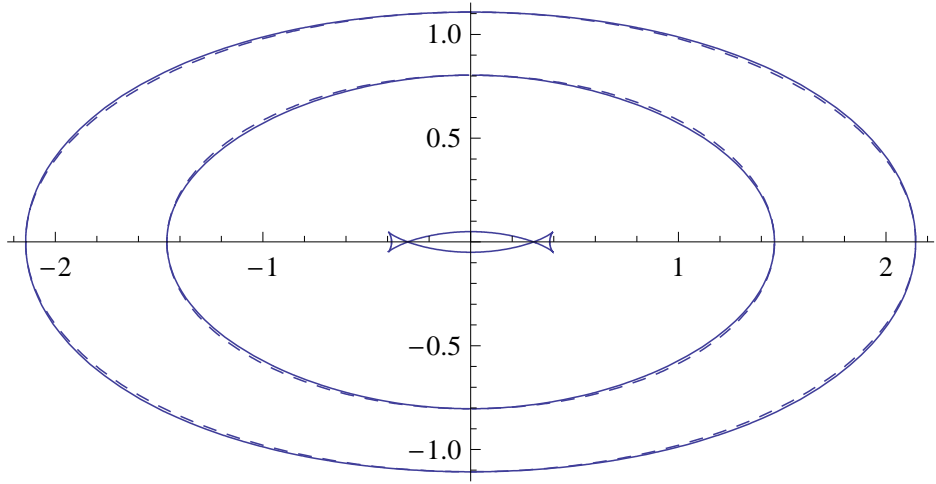}
%{figure6.eps}
%%{ellipticity.eps}
\caption{$n=6,A_1=1.25, A_2=1.5$. The components of $C(A)$ are nested, with $\gamma_1$ and $\gamma_2$ being convex and so coinciding with the boundaries of $\Lambda_1,\Lambda_2$, respectively. On the other hand, $\Lambda_3$ is bounded by the ``middle portion'' of $\gamma_3$.}

\label{fig6x6}
\end{figure}

\begin{figure}[p]
\centering
\includegraphics[width=0.9\linewidth,angle=0]{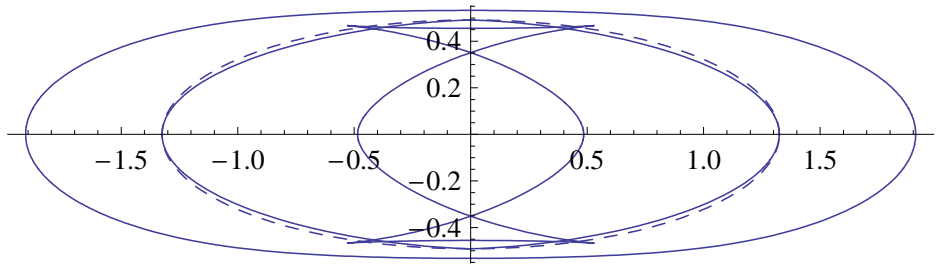}
%{figure6.eps}
%%{ellipticity.eps}
\caption{$n=6,A_1=1.05, A_2=1.62$. The component $\gamma_1$ and $\gamma_2$ are still convex. As opposed to Fig.~\ref{fig6x6}, $\gamma_3$ is intersecting with $\gamma_2$. }
%}}
\label{fig6N}
\end{figure}
\begin{figure}[p]
\centering
\includegraphics[width=0.9\linewidth,angle=0]{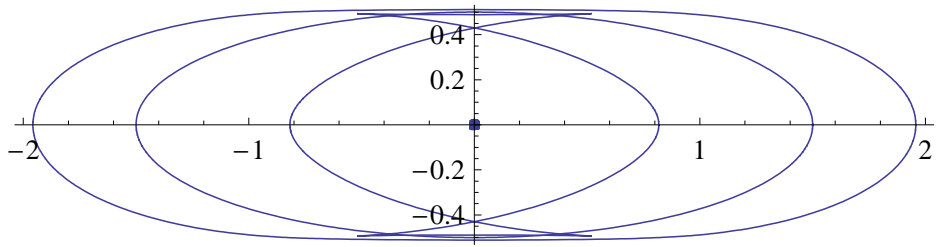}
%{figure6.eps}
%%{ellipticity.eps}
\caption{$n=7,A_1=1.05,A_2=1.62$. The picture is similar to Fig.~\ref{fig6N}, except that $\gamma_2$ is an exact ellipse, and there emerges $\gamma_4=\{0\}$. }
%}}
\label{fig7N}
\end{figure}

\begin{figure}[t]
\centering
\includegraphics[width=0.9\linewidth,angle=0]{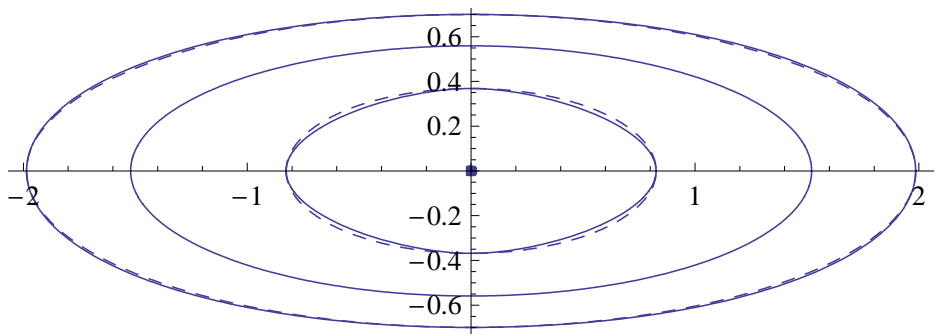}
%{figure6.eps}
%%{ellipticity.eps}
\caption{$n=7, A_1=2, A_2=1.5$. The components $\gamma_j$ are convex for $j=1,2,3$ and visually indistinguishable from ellipses, though only the middle one is a genuine ellipse.}
%}}
\label{fig7NC}
\end{figure}

\newpage

\providecommand{\bysame}{\leavevmode\hbox to3em{\hrulefill}\thinspace}
\providecommand{\MR}{\relax\ifhmode\unskip\space\fi MR }
\providecommand{\MRhref}[2]{%
	\href{http://www.ams.org/mathscinet-getitem?mr=#1}{#2}
}
\providecommand{\href}[2]{#2}

\end{document}